\documentclass[12pt]{amsart}
\usepackage{amsmath,amssymb,amsfonts,amsthm,amstext,verbatim,url}
\usepackage{graphicx}
\RequirePackage[colorlinks,citecolor=blue,urlcolor=blue]{hyperref}
\usepackage{tikz}
\usetikzlibrary{graphs}

\theoremstyle{plain}

\newtheorem{theorem}{Theorem}
\newtheorem{remark}[theorem]{Remark}
\newtheorem{definition}[theorem]{Definition}

\newtheorem{example}[theorem]{Example}

\numberwithin{equation}{section}
\numberwithin{theorem}{section}
\numberwithin{figure}{section}

\newcounter{mycount}

\newenvironment{letlist}{\begin{list}{\rm(\alph{mycount})}%
   {\usecounter{mycount}\labelwidth=1cm\itemsep 0pt}}{\end{list}}

\newcommand\s{\sigma}

\newcommand\oo{\infty}

\newcommand\HH{{\mathbb H}}
\newcommand\NN{{\mathbb N}}

\newcommand\sA{{\mathcal A}}

\newcommand\sAc{\sA_{\text{\rm c}}}

\newcommand\ZZ{{\mathbb Z}}
\newcommand\RR{{\mathbb R}}

\newcommand\wt{\widetilde}

\renewcommand\a{\alpha}

\newcommand\Si{\Sigma}

\newcommand\g{\gamma}
\newcommand\resp{respectively}

\begin{document}
\title[SAWs on cubic graphs]
{Self-avoiding walks on cubic graphs and local transformations}
\author[Grant]{Benjamin Grant}
\address{Department of Mathematics, University of Connecticut} 
\email{\{benjamin.grant, zhongyang.li\}@uconn.edu}
\urladdr{https://sites.google.com/view/benjamingrant}
\urladdr{https://mathzhongyangli.wordpress.com/}

\author[Li]{Zhongyang Li}

\begin{abstract}
Despite its elementary definition, the self-avoiding walk (SAW) poses notoriously hard
enumerative problems: exact connective constants are known for only a handful of infinite
graphs, notably the honeycomb lattice \cite{ds}.
We establish a general substitution principle for SAWs on infinite connected quasi-transitive
cubic graphs under port-transitive vertex replacements, where each degree-$3$ vertex is
replaced by a fixed finite three-port gadget.
Writing $g(x)$ for the associated two-port SAW series, we prove that for $G_1=\phi(G)$,
\[
\mu(G)^{-1}=g\bigl(\mu(G_1)^{-1}\bigr),
\]
equivalently $\mu(G_1)^{-1}$ is the unique solution $x\in(0,1)$ of $g(x)=\mu(G)^{-1}$,
thereby extending the Fisher-triangle relation of Grimmett--Li to arbitrary symmetric
three-port gadgets.
We also obtain the corresponding identity for bipartite graphs when one or both colour
classes are transformed, and show that the critical exponents $\gamma$ and $\eta$ (and
$\nu$ under a standard regularity hypothesis) are invariant.
For explicit gadget families, including complete-graph gadgets $K_N$ and Fisher-type
constructions, these identities turn base graphs with known $\mu$ into infinite families
of new quasi-transitive graphs whose connective constants are determined exactly as the
unique roots of explicit algebraic equations.
\end{abstract}

\date{20 Jan 2026} 


\maketitle

\section{Introduction}\label{sec:intro}

A \emph{self-avoiding walk} (SAW) on a graph is a finite path that visits no vertex more
than once.  Introduced as a model for polymer chains, SAWs have become central objects
in probability, combinatorics, and statistical physics; see, for example, \cite{f,ms}.
Let $G$ be an infinite connected quasi-transitive graph, and let $\sigma_n(v)$ denote the
number of $n$-step SAWs starting at $v$.  A classical theorem of Hammersley implies that
the limit
\[
\mu(G)=\lim_{n\to\infty}\sigma_n(v)^{1/n}
\]
exists and is independent of $v$; the constant $\mu(G)$ is the \emph{connective constant}
of $G$.  Exact values (or explicit characterizations) of $\mu(G)$ are known only in a small
number of nontrivial cases.

\medskip
\noindent\textbf{Local vertex replacements and a substitution principle.}
We study SAWs on cubic graphs under a class of vertex-replacement operations.
A \emph{local transformation} $\phi$ replaces each degree-$3$ vertex by a fixed finite
\emph{three-port gadget} whose three ports attach to the incident edges, subject to a
port-transitivity condition on the gadget automorphism group (Definition~\ref{df31}).
The classical example is the \emph{Fisher transformation}, where each vertex is replaced
by a triangle.

Grimmett and Li \cite{GrimmettLi2013Fisher} analysed SAWs under the Fisher transformation,
obtaining an explicit functional relation between connective constants and proving
invariance of key critical exponents under iterated Fisher transforms.
The aim of the present paper is to establish a general \emph{exact substitution framework} that
(i) extends the Fisher-triangle analysis to \emph{arbitrary} port-transitive three-port gadgets,
and (ii) treats the \emph{bipartite} setting in which the transformation is applied to one
colour class, or to both classes with possibly different gadgets.  The resulting identities
are global (they describe radii of convergence and critical exponents) but are driven by
finite, explicitly computable local data.

\medskip
\noindent\textbf{A two-port gadget series and connective constants.}
Fix $\phi$, and let $H$ denote its gadget.  We associate to $H$ a two-port generating function
\begin{equation}\label{eq:intro-g}
g(x)=\sum_{\pi} x^{|\pi|},
\end{equation}
where $\pi$ ranges over SAWs that enter $H$ through one prescribed port-edge and exit
through another while otherwise remaining inside $H$ (Definition~\ref{df32}).
Our first main theorem (Theorem~\ref{thm:main2}) shows that this purely local series governs
the connective constant after transforming \emph{every} vertex:
if $G$ is infinite, connected, quasi-transitive, and cubic, and $G_1=\phi(G)$, then
\begin{equation}\label{eq:intro-mu}
\mu(G)^{-1}=g\bigl(\mu(G_1)^{-1}\bigr),
\end{equation}
equivalently $\mu(G_1)^{-1}$ is the unique $x\in(0,1)$ solving $g(x)=\mu(G)^{-1}$.

\medskip
\noindent\textbf{Critical exponents on quasi-transitive graphs.}
To avoid assuming an ambient Euclidean dimension, we adopt graph-theoretic definitions of
critical exponents $\gamma,\eta,\nu$ via finiteness of natural susceptibility-type series
(Section~\ref{sec:fisher}).  Our second main theorem (Theorem~\ref{thm:main2'}) proves that
local transformations preserve these exponents under mild hypotheses: $\gamma$ and $\eta$
are invariant in general, and $\nu$ is invariant provided the SAW counts satisfy a standard
regularity hypothesis with a common slowly varying correction.

\medskip
\noindent\textbf{Bipartite transformations.}
If $G$ is bipartite and the transformation is applied to one colour class (or, in the fully
cubic case, to both classes with possibly different gadgets), we obtain an analogous
relation (Theorem~\ref{thm:main3}):
\[
\mu(G)^{-2}=h\bigl(\mu(G_{\mathrm e})^{-1}\bigr),
\]
with $h(x)=xg(x)$ in the one-class case and $h(x)=g_{\phi_1}(x)\,g_{\phi_2}(x)$ in the two-class case.
These formulas turn known values of $\mu(G)$ into explicit equations for $\mu(G_{\mathrm e})$.

\medskip
\noindent\textbf{Examples and computations.}
Section~\ref{sec:examples} develops explicit gadget families and worked base-graph examples.
We emphasize complete-graph gadgets $K_N$ (yielding closed-form polynomials $g_N$), a mixed
bipartite transformation on the hexagonal lattice (white $K_4$, black Fisher triangles),
and an iterated transformation example that exhibits convergence of the sequence of
connective constants to a gadget-determined fixed point.

\medskip
\noindent\textbf{Organization.}
Section~\ref{sec:notation} introduces notation.
Section~\ref{sec:fisher} defines local transformations and the associated series $g(x)$ and
states Theorems~\ref{thm:main2}--\ref{thm:main3}.
Proofs appear in Sections~\ref{sec:main2} and~\ref{sec:proof3}.
Section~\ref{sec:examples} contains explicit examples, computations, and remarks on iteration.

\section{Notation}\label{sec:notation}

All graphs studied henceforth in this paper will be assumed infinite, connected,
and simple (in that they have neither loops nor multiple edges). 
An edge $e$ with endpoints $u$, $v$ is written $e=\langle u,v \rangle$.
If $\langle u,v \rangle \in E$, we call $u$ and $v$ \emph{adjacent}
and write $u \sim v$.
The \emph{degree} of vertex $v$ is the number of edges
incident to $v$, denoted $\deg(v)$. A graph is called \emph{cubic}
if all vertices have degree $3$.
The \emph{graph-distance} between two vertices $u$, $v$ is the number of edges
in the shortest path from $u$ to $v$, denoted $d_G(u,v)$.

The automorphism group of the graph $G=(V,E)$ is
denoted $\sA=\sA(G)$.
The graph $G$ is
called \emph{quasi-transitive} if there exists a finite subset $W \subseteq V$ such that,
for $v \in V$ there exists $\a\in\sA $ such
that $\a v \in W$. We call such $W$ a \emph{fundamental domain},
and shall normally (but not invariably) take $W$
to be minimal with this property. The graph is 
called \emph{vertex-transitive} 
(or \emph{transitive}) if the singleton set $\{v\}$ is a fundamental
domain for some (and hence all) $v \in V$.

A \emph{walk} $w$ on $G$ is
an alternating sequence $v_0e_0v_1e_1\cdots e_{n-1} v_n$ of vertices $v_i$
and edges $e_i$ such that $e_i=\langle v_i, v_{i+1}\rangle$.
We write $|w|=n$ for the \emph{length} of $w$, that is, the number of edges in $w$.

Let $n \in \NN$, the natural numbers. 
A $n$-step \emph{self-avoiding walk} (SAW) 
on $G$ is  a walk containing $n$ edges
that includes no vertex more than once.
Let $\s_n(v)$ be the number of $n$-step SAWs 
 starting at $v\in V$. It was shown by Hammersley \cite{jmhII}
that, if $G$ is quasi-transitive, there exists a constant $\mu=\mu(G)$,
called the \emph{connective constant} of $G$,
such that
\begin{equation}
\label{connconst}
\mu= \lim_{n\to\oo}  \s_n(v)^{1/n}, \qquad v \in V.
\end{equation}

It will be convenient to consider also SAWs starting at `mid-edges'. We identify the
edge $e$ with a point (also denoted $e$) placed at the middle of $e$, 
and then consider walks that start and end at 
these mid-edges. Such a  walk is \emph{self-avoiding} if it visits no vertex or mid-edge 
more than once, and its \emph{length} is the number of vertices visited.

The minimum of two reals $x$, $y$ is denoted $x \wedge y$, and the maximum
$x \vee y$.

\section{Local transformation} \label{sec:fisher}

\begin{definition}[Local transformation]\label{df31}
Let $G=(V,E)$ be a simple graph, and let $v\in V$ be a vertex of degree $3$.
A \emph{local transformation} $\phi$ acts at $v$ by replacing $v$ with a finite
graph (a \emph{gadget}) $\phi(v)$ that satisfies the following conditions:
\begin{enumerate}
\item The gadget $\phi(v)$ comes with three distinguished \emph{outer vertices},
      each of which is attached to exactly one of the three edges incident to $v$
      (one outer vertex per incident edge).
\item For any degree-$3$ vertices $v_1,v_2\in V$, the gadgets $\phi(v_1)$ and $\phi(v_2)$
      are isomorphic (via an isomorphism that sends outer vertices to outer vertices).
\item Let $w_1,w_2,w_3$ be the three outer vertices of $\phi(v)$. There exists a subgroup of the automorphism group
      of the gadget that
    \begin{itemize}  
\item       acting transitively on $\{w_1,w_2,w_3\}$; equivalently, for any
      $w_i,w_j\in\{w_1,w_2,w_3\}$ there exists an automorphism of $\phi(v)$ mapping $w_i$ to $w_j$; and
      \item preserving the set of distinguished outer vertices $\{w_1,w_2,w_3\}$
      \end{itemize}
\end{enumerate}
The resulting graph is denoted by $\phi(G)$. 
\end{definition}

One special case of local transformations defined above is the so-called Fisher transformation; in which each vertex is replaced by a triangle.
This transformation has been valuable in the study of
the relations between Ising, dimer, and general vertex models 
(see \cite{bdet,fisher,zli,li}), and more recently of SAWs
on the Archimedean lattice denoted $(3,12^2)$ (see \cite{g,jg}).

In the remainder of this paper, we make use of the more general local transformation in the context of 
SAWs and the connective constant. It will be applied to cubic graphs, of which the 
hexagonal and square/octagon lattices are examples.

It is convenient to work with graphs with well-defined connective constants, and
to this end we assume 
that $G=(V,E)$ is quasi-transitive and connected, so that its connective constant
is given by \eqref{connconst}.
We write $\phi(G)$ for the graph obtained
from the cubic graph $G$ by applying the local transformation at every vertex. 
Obviously the automorphism group of $G$ induces an automorphism subgroup
of $\phi(G)$. 

\begin{definition}\label{df32}
Let $G=(V,E)$ be a simple graph. For a degree-$3$ vertex $v\in V$, let $\phi$ be a local
transformation as in Definition~\ref{df31}. Let $e_1,e_2,e_3$ be the three edges of $G$
incident to $v$. Define
\begin{equation}\label{dgx}
g(x)\;=g_{\phi}(x)=\;\sum_{\pi\in \Sigma_{\phi(v),e_1,e_2}} x^{|\pi|},
\end{equation}
where $\Sigma_{\phi(v),e_1,e_2}$ denotes the set of all self-avoiding walks (SAWs) that
start at the midpoint of $e_1$, end at the midpoint of $e_2$, and otherwise visit only
vertices of the gadget $\phi(v)$. $|\pi|$ denote the total number of vertices in the SAW $\pi$.

By Definition~\ref{df31}(2), the gadgets $\phi(v)$ are mutually isomorphic for all
degree-$3$ vertices $v$, so the power series $g(x)$ is independent of the choice of $v$.
Moreover, by Definition~\ref{df31}(3), the value of $g(x)$ is the same if
$\{e_1,e_2\}$ is replaced by any other pair among
$\{e_1,e_3\}$ and $\{e_2,e_3\}$ in~\eqref{dgx}.
\end{definition}

\begin{theorem}\label{thm:main2}
Let $G$ be an infinite, quasi-transitive, connected, cubic graph,  $G_1 = \phi(G)$.
Then
 The connective constants $\mu_1$ of $G_1$ and $\mu$ of $G$ satisfy
$\mu^{-1} = g(\mu_{1}^{-1})$ where $g(x)$ is defined as in (\ref{dgx}).
\end{theorem}

We turn to the topic of \emph{critical exponents}, beginning with
a general introduction for the case when there exists a periodic, locally finite embedding
of $G$ into $\RR^d$ with $d \ge 2$. The case of general $G$ has not not been 
studied extensively, and most attention has been paid to the 
hypercubic lattice $\ZZ^d$.
It is believed 
 (when $d \ne 4$) that there is a power-order correction, 
in the sense that there exists $A_v>0$ and an exponent 
$\g \in \RR$ such that
\begin{equation}\label{defgamma}
\s_n(v)\sim A_v n^{\gamma-1}\mu^n \qquad \text{as } n \to\oo,
\qquad v \in V.
\end{equation}
Furthermore, 
the value of the exponent $\g$ is believed to depend on $d$ and not further on the 
choice of graph $G$.
When $d=4$, \eqref{defgamma} should hold with $\gamma=1$
and subject to the inclusion
on the right side of the logarithmic correction
term  $(\log n)^{1/4}$.
See \cite{bdgs,ms} for  accounts of critical exponents for SAWs.

Let $v \in V$ and
\begin{equation}
Z_{v,w}(x)=\sum_{n=0}^\oo \s_n(v,w)x^k, \qquad w \in V,\ x>0,
\label{defz}
\end{equation} 
where $\s_n(v,w)$ is the number of $n$-step SAWs with endpoints
$v$, $w$.
It is known under certain circumstances that the generating 
functions $Z_{v,w}$ have radius of convergence $\mu^{-1}$ 
(see \cite[Cor.\ 3.2.6]{ms}),
and it is believed that there exists an exponent $\eta$ and constants 
$A'_v>0$  such that
\begin{equation}\label{defeta}
Z_{v,w}(\mu^{-1})\sim A_v'{d_G(v,w)}^{-(d-2+\eta)}
\qquad \text{as } d_G(v,w)\to\oo.
\end{equation}

Let $\Si_n(v)$ be the set of $n$-step SAWs from $v$,
and write $\langle \cdot \rangle_n^v$ for expectation with respect to
uniform measure on $\Sigma_n(v)$. Let
$\|\pi\|$ be the graph-distance between the endpoints of 
a SAW $\pi$.
It is believed (when $d\ne 4$) that
there exists an exponent $\nu$ and constants
$A''_v>0$ such that  
\begin{equation}\label{defnu}
\langle \|\pi\|^2\rangle_n^v \sim A_v'' n^{2\nu}, \qquad v \in V.
\end{equation}
As above, this should hold for $d=4$ with $\nu=\frac12$ and subject to the
inclusion of the correction term $(\log n)^{1/4}$.

The above three exponents  are believed to be related to one another 
through the so-called \emph{Fisher relation}
\begin{equation}\label{fisher-rel}
\gamma = \nu(2-\eta).
\end{equation}

It is convenient to work with 
definitions of critical exponents that do not depend on
an assumption of dimensionality, and thus we proceed as follows.
Let $G$ be an infinite, connected, quasi-transitive graph with 
connective constant $\mu$ and fundamental domain $W$.
Let $X$ be the set of edges incident to vertices in  $W$,
and let $\Sigma$ be the set of SAWs on $G$ starting at mid-edges in $X$.
We define the function
\begin{equation*}\label{expatn}
Y(x,y)=
\sum_{\pi\in\Sigma}
\frac{x^{|\pi|}}{|\pi|^y}, \qquad x> 0,\ y\in\RR.
\end{equation*}
(The denominator is interpreted as 1 when $|\pi|=0$.)
For fixed $x$, $Y(x,y)$ is non-increasing in $y$. 
Let $\g=\g(G)\in[-\oo,\oo]$ be such that
$$
Y(\mu^{-1},y) \begin{cases} =\infty &\text{if } y<\gamma,\\ 
<\oo &\text{if } y>\g.
\end{cases}
$$
We shall assume that $-\oo < \g<\oo$. 
It will be convenient at times to  assume more about
the number $\s_n$ of $n$-step SAWs from mid-edges in $X$, 
namely that   there exist constants $C_i=C_i(W) \in(0,\infty)$ and a slowly varying function $L$ 
such that
\begin{equation}\label{def:slow}
C_1 L(n)n^{\gamma-1}\mu^n\leq \sigma_{n}\leq C_2 L(n)n^{\gamma-1}\mu^n, \qquad n \ge 1.
\end{equation}

Let
\begin{equation}\label{def0}
V(z)=\sum_{n=1}^{\infty}\frac{1}{n^{2z+1}} \langle \|\pi\|^2 \rangle_n,
\qquad z \in [-\oo,\oo],
\end{equation}
where $\langle \cdot \rangle_n$ denotes the uniform average over the
set $\Sigma_n$ of $n$-step SAWs in $\Sigma$.
Thus, $V(z)$ is non-increasing in $z$, and we let  
$\nu=\nu(G)\in[-\oo,\oo]$ be such that
$$
V(z) \begin{cases} =\infty &\text{if } z < \nu,\\
<\infty &\text{if } z > \nu.
\end{cases}
$$

Let $\a W$ denote the image of $W$ under an automorphism $\a\in\sA$, with incident
edges $\a X$, and let
\begin{equation*}
Z_\a(x)=\sum_{\pi\in\Sigma(\a)} x^{|\pi|},
\end{equation*}
where $\Sigma(\a)$ is the subset of $\Sigma$ containing SAWs ending 
at mid-edges in $\a X$.
We assume there exists $\eta =\eta(G) \in [-\oo,\oo]$ such that,
for any sequence of automorphisms $\a$ satisfying $d_G(W,\a W)\to\oo$,
\begin{equation}\label{dd1}
Z_\a(\mu^{-1})d_G(W,\a W)^{w} 
\begin{cases} \to 0 &\text{if } w<\eta,\\
\to\infty &\text{if } w>\eta. 
\end{cases}
\end{equation}
The $\eta$ of \eqref{defeta} should agree with that defined here when $d=2$.

It is easily seen that the values of $\g$, $\eta$, $\nu$ do not depend on the choice of
fundamental domain $W$. 

We consider now the effect on critical exponents of
the local transformation. 
Let $W_0$ be a minimal fundamental domain of $G_0 := G$, 
with incident edge-set $X_0 := X$ as above. Write $W_1=F(W_0)$, 
the set of vertices of the 
triangles formed by the Fisher transformation at vertices in $W_0$,
and $X_1$ for the set of edges of $G_1$ incident to vertices in $W_1$.
It may be seen that $W_1$ is a fundamental domain of $G_1$.

\begin{theorem}\label{thm:main2'}
Let $G_0$ be an infinite, quasi-transitive, connected, cubic graph.
Assume that $|\g(G_0)|<\oo$ and that $\eta(G_1)$ exists.
\begin{letlist}
\item The exponents $\g$, $\eta$ of $G_0$ and $G_1$ are equal.
\item  Let $\sigma_{n,k}$ be
the number of $n$-step SAWs on $G_k$ from mid-edges in $X_k$.
Assume the $\s_{n,k}$ satisfy \eqref{def:slow} for constants
$C_{i,k}$ and a common slowly varying function $L$. Then
the exponents $\nu$ of $G_0$ and $G_1$ are equal.
\end{letlist}
\end{theorem}

Our final result concerns the effect of the local transformation 
when applied to just one of the vertex-sets of a bipartite graph.
Let $G=(V,E)$ be bipartite with vertex-sets $V_1$, $V_2$ coloured white and black, 
\resp. We think of $G$ as a graph together with a colouring $\chi$, and 
the \emph{coloured-automorphism group} $\sAc=\sAc(G)$ of the pair $(G,\chi)$ is the set
of maps $\phi:V \to V$ which preserve both graph structure and colouring.
The coloured graph is \emph{quasi-transitive} if there exists a finite 
subset $W \subseteq V$ such
that: for all $v\in V$, there exists $\a\in\sAc$ such that $\a v\in W$ and $\chi(v) = \chi(\a v)$. 
As before, such a set $W$ is called a \emph{fundamental domain}.

\begin{theorem}\label{thm:main3}
\begin{enumerate}
\item Let $G$ be an infinite,  connected,
bipartite graph with vertex-sets coloured black and white, and suppose that
the coloured graph $G$ is quasi-transitive, and every black vertex has degree $3$. 
Let $\wt G$ be obtained by applying the local transformation at each black vertex. 
\begin{letlist}
\item The connective constants $\mu$ and $\wt\mu$ of $G$ and $\wt G$, \resp, satisfy 
$\mu^{-2} = h(\wt\mu^{-1})$ where
$h(x) =xg(x)$.
\item 
Under the corresponding assumptions of Theorem \ref{thm:main2'}, the exponents $\g$, $\eta$, $\nu$ are the same for $G$ 
as for $\wt G$.
\end{letlist}
\item Let $G$ be an infinite,  connected,
bipartite graph with vertex-sets coloured black and white, and suppose that
the coloured graph $G$ is quasi-transitive, and every vertex has degree $3$. 
Let $\wt G$ be obtained by applying the local transformation $\phi_1$ at each black vertex; and local transformation $\phi_2$ at each white vertex
\begin{letlist}
\item The connective constants $\mu$ and $\wt\mu$ of $G$ and $\wt G$, \resp, satisfy 
$\mu^{-2} = h(\wt\mu^{-1})$ where
$h(x) =g_{\phi_1}(x)g_{\phi_2}(x)$.
\item 
Under the corresponding assumptions of Theorem \ref{thm:main2'}, the exponents $\g$, $\eta$, $\nu$ are the same for $G$ 
as for $\wt G$.
\end{letlist}
\end{enumerate}
\end{theorem}
 
Theorem \ref{thm:main3}(a) implies an exact value of a connective constant that does
not appear to have been noted previously. Take $G=\HH$, the hexagonal
lattice with connective constant $\mu= \sqrt{2+\sqrt 2} \approx 1.84776$,
see \cite{ds}.

The proofs of Theorems \ref{thm:main2}-\ref{thm:main2'} and 
\ref{thm:main3} are found  in
Sections \ref{sec:main2} and \ref{sec:proof3}, \resp.

\section{Proof of Theorems \ref{thm:main2}-\ref{thm:main2'}}\label{sec:main2}

\begin{proof}[Proof of Theorem \ref{thm:main2}]

Let $G=(V,E)$ be an infinite, connected, quasi-transitive, cubic graph. 
The graph $G_1 = \phi(G)$ is also quasi-transitive. It suffices for part (a) to show that the connective constants $\mu_k$
of the $G_k$ satisfy 
\begin{equation}\label{10}
g(\mu_1^{-1}) = \mu_0^{-1}.
\end{equation}

Let $W$ be a minimal fundamental domain of $G$, and let 
$X$ be the subset of $E$ comprising all edges incident 
to vertices in $W$. Write $W_1=\phi(W)$, the set of vertices of the 
triangles formed by the Fisher transformation at vertices in $W$,
and $X_1$ for the set of edges of $G_1$ incident to vertices in $W_1$.
It may be seen that $W_1$ is a fundamental domain of $G_1$.

It is convenient to work with SAWs that start and end at mid-edges.
Note that the mid-edges 
of $E$ (\resp, $X$) may be viewed as  mid-edges of 
$E_1$ (\resp, $X_1$). 

Let $G_0:=G$. For $k=0,1$, the partition functions of SAWs on $G_k$ are the polynomials
\begin{equation*}
Z_k(x)=\sum_{\pi\in \Sigma_k} x^{|\pi|}, \qquad x >0,
\end{equation*}
where the sum is over the set $\Sigma_k$ of
SAWs starting at mid-edges of $X_k$. Similarly, we set
\begin{equation*}
Z_1^*(x) = \sum_{\pi \in \Sigma_1^*} x^{|\pi|},
\end{equation*}
where the sum is over the set $\Sigma_1^*$ of SAWs on 
$G_1$ starting at mid-edges of $X$ and ending at mid-edges of $E$.
For $k=0,1$,
\begin{equation}\label{225}
Z_k(x) 
\begin{cases} <\oo &\text{if } x < \mu_k^{-1},\\
=\oo &\text{if } x > \mu_k^{-1}.
\end{cases}
\end{equation}

The following basic argument
formalizes a method known already in the special case of the 
hexagonal lattice, see for example \cite{g,jg}.
Since $\Sigma_1^* \subseteq \Sigma_1$, we have 
\begin{equation}\label{221}
Z_1^*(x) \le Z_1(x).
\end{equation}
Let  $N(\pi)$ be the number of endpoints of a SAW 
$\pi\in \Sigma_1$ 
that are mid-edges of $E$. 
The set $\Sigma_1$ may be partitioned into three sets.
\begin{letlist}
\item If $N(\pi)=2$, then $\pi$ contributes
to $Z_1^*$.
\item $\pi$ may be a  walk within a single gadget.
\item If (b) does not hold and $N(\pi)\le 1$, any endpoint not in $E$ may
be moved by finitely many steps (indeed, at most $|\phi(v)|$, where $|\phi(v)|$ is the number of vertices in the gadget $\phi(v)$) along $\pi$ to obtain a shorter SAW in $\Sigma_1^*$.
\end{letlist}
By considering the numbers of SAWs in each subcase of (c), we find that
\begin{equation}\label{222}
Z_1(x) \leq  
\left[1+\left(\sum_{i=1}^{|\phi(v)|}(Dx)^i\right)\right]^2 Z_1^*(x) +|E(\phi(v))||W|\left(\sum_{i=1}^{|\phi(v)|}(Dx)^i
\right).
\end{equation}
where the last term corresponds to case (b) and $|E(\phi(v))|$ is the number of edges in the gadget $\phi(v)$, and $D$ is the maximal vertex degree in $\phi(v)$.
By \eqref{221}--\eqref{222}, 
$$
Z_1(x)<\oo \quad\Leftrightarrow\quad Z_1^*(x)<\oo,
$$
so that, by \eqref{225}, 
\begin{equation}\label{223}
Z_1^*(x) \begin{cases} <\oo &\text{if } x < \mu_1^{-1},\\
=\oo &\text{if } x > \mu_1^{-1}.
\end{cases}
\end{equation}

With a SAW in $\Sigma_1^*$ we associate a SAW in $\Sigma_0$ 
by shrinking each gadget $\phi(v)$ to a vertex. 

Let $r$ be the number of SAW's in a gadget $\phi(v)$ joining two outer vertices; by Definition \ref{df31}(c), $r$ is independent of the two specific outer vertices chosen.
Each $n$-step SAW  
in $\Sigma_0$ arises thus from $r^n$ SAWs in $\Sigma_1^*$. 
Therefore,
\begin{equation}\label{r2}
Z_0(g(x))=Z_{1}^*(x),
\end{equation}
and \eqref{10} follows by \eqref{225} and \eqref{223}.
\end{proof}

\begin{proof}[Proof of Theorem~\ref{thm:main2'}]
Throughout, let $G_0:=G$ and $G_1:=\phi(G)$, and let $W_k,X_k,\Sigma_k$ be as in
Section~3.  Let $\Sigma_1^\ast$ be the set of SAWs on $G_1$ that start at mid-edges of $X_0$
and end at mid-edges of $E_0$ (viewed as mid-edges of $E_1$).  For $n\ge 0$, write
$\Sigma_{n,k}$ (resp.\ $\Sigma_{n,1}^\ast$) for the subset of $\Sigma_k$ (resp.\ $\Sigma_1^\ast$)
containing $n$-step SAWs, and let $\sigma_{n,k}:=|\Sigma_{n,k}|$ and $\sigma_{n,1}^\ast:=|\Sigma_{n,1}^\ast|$.

Let $N:=|V(\phi(v))|$ be the number of vertices of a gadget (independent of $v$ by
Definition~3.1(2)).  Let $\Delta:=\mathrm{diam}(\phi(v))$ and $D:=\max_{u\in V(\phi(v))}\deg_{\phi(v)}(u)$
(both independent of $v$).  Note that any SAW segment inside a gadget connecting two
ports has length between $2$ and $N$ (in the mid-edge convention of Section~2).

\medskip\noindent
\textbf{(a) Equality of $\gamma$ and $\eta$.}
For $k=0,1$ define
\[
Y_k(x,y)\;=\;\sum_{\pi\in\Sigma_k}\frac{x^{|\pi|}}{|\pi|^{\,y}},
\qquad
Y_1^\ast(x,y)\;=\;\sum_{\pi\in\Sigma_1^\ast}\frac{x^{|\pi|}}{|\pi|^{\,y}},
\qquad x>0,\ y\in\mathbb{R},
\]
with the convention that the denominator is $1$ when $|\pi|=0$.

\smallskip\noindent
\emph{Step 1: $Y_1^\ast(x,y)<\infty \iff Y_1(x,y)<\infty$ for fixed $(x,y)$.}
Since $\Sigma_1^\ast\subseteq \Sigma_1$, we have $Y_1^\ast(x,y)\le Y_1(x,y)$.

Conversely, consider $\pi\in\Sigma_1\setminus\Sigma_1^\ast$.  Either:
(i) $\pi$ is entirely contained in a single gadget (there are only finitely many such walks,
of bounded length $\le N$), or
(ii) $\pi$ visited at least one mid-edge in $E_0$ and at least one endpoint of $\pi$ is a mid-edge not in $E_0$.
In case (ii), move each endpoint that is not in $E_0$ along $\pi$ to the first mid-edge of $E_0$
encountered when traversing $\pi$ from that endpoint.  This removes at most $N$ visited vertices
per endpoint, hence produces a (strictly) shorter SAW $\pi^\dagger\in \Sigma_1^\ast$ with
$|\pi|-2N\le |\pi^\dagger|\le |\pi|$.

Moreover, for a fixed $\pi^\dagger$, the number of possible $\pi$ that reduce to it by this procedure
is bounded by a constant depending only on $(N,D)$ (each endpoint extension has at most $\sum_{j=0}^{N}(Dx)^j$
choices at weight $x$, and there are two endpoints).  Since $|\pi|$ and $|\pi^\dagger|$ differ by at most $2N$,
we also have for fixed $y\in\mathbb{R}$ a constant $C_y<\infty$ such that
$|\pi|^{-y}\le C_y\,|\pi^\dagger|^{-y}$.
It follows that for each fixed $(x,y)$,
\[
Y_1(x,y)\;\le\;C(x,y)\,Y_1^\ast(x,y)\;+\;C'(x,y),
\]
where $C'(x,y)<\infty$ accounts for the finitely many SAWs contained in a single gadget.
Therefore,
\begin{equation}\label{eq:Ystar-equiv}
Y_1^\ast(x,y)<\infty \iff Y_1(x,y)<\infty .
\end{equation}

\smallskip\noindent
\emph{Step 2: compare $Y_1^\ast(x,y)$ and $Y_0(g(x),y)$.}
Fix an $n$-step SAW $\omega\in\Sigma_{n,0}$.  Contracting each gadget of $G_1$ to a vertex
maps every $\pi\in\Sigma_1^\ast$ to a unique $\omega\in\Sigma_0$; conversely, expanding each visit
to a vertex of $\omega$ by a SAW segment inside the corresponding gadget yields SAWs in $\Sigma_1^\ast$.
Let $a_m:=\sigma^{\phi(v)}_{m}(e_1,e_2)$ be the number of $m$-step gadget SAWs (in the mid-edge convention)
connecting two ports, so that
\[
g(x)=\sum_{m\ge 0} a_m x^{m},
\qquad a_0=a_1=0,\qquad a_m=0 \text{ for } m>N.
\]
Let $b_{n,m}$ be the coefficient of $x^m$ in $g(x)^n$, i.e.\ $g(x)^n=\sum_m b_{n,m}x^m$.
Then $b_{n,m}\neq 0$ only when $2n\le m\le Nn$, and the total contribution to $Y_1^\ast(x,y)$
from all $\pi\in\Sigma_1^\ast$ that contract to this fixed $\omega$ equals
\[
T_n(x,y)\;:=\;\sum_{m=2n}^{Nn} b_{n,m}\,\frac{x^{m}}{m^{\,y}}.
\]
Since $2n\le m\le Nn$, we have
\[
(2n)^{-y}\wedge (Nn)^{-y}\;\le\; m^{-y}\;\le\;(2n)^{-y}\vee (Nn)^{-y}.
\]
Hence, with
\[
c_y := 2^{-y}\wedge N^{-y},
\qquad
c_y' := 2^{-y}\vee N^{-y},
\]
we obtain the bounds
\[
c_y\,\frac{g(x)^n}{n^{\,y}}
\;\le\;
T_n(x,y)
\;\le\;
c_y'\,\frac{g(x)^n}{n^{\,y}}.
\]
Summing $T_n(x,y)$ over all $\omega\in\Sigma_{n,0}$ and then over $n\ge 1$ gives
\begin{equation}\label{eq:Ystar-Y0-compare}
c_y\,\widetilde Y_0(g(x),y)\;\le\;\widetilde Y_1^\ast(x,y)\;\le\;c_y'\,\widetilde Y_0(g(x),y),
\end{equation}
where $\widetilde Y$ denotes the series with the (harmless) $n=0$ term removed.
In particular,
\begin{equation}\label{eq:Ystar-finite-iff}
Y_1^\ast(x,y)<\infty \iff Y_0(g(x),y)<\infty .
\end{equation}

By Theorem~3.3, $\mu_0^{-1}=g(\mu_1^{-1})$.  Combining \eqref{eq:Ystar-equiv} and \eqref{eq:Ystar-finite-iff}
at $x=\mu_1^{-1}$ yields
\[
Y_1(\mu_1^{-1},y)<\infty \iff Y_0(\mu_0^{-1},y)<\infty,
\]
and therefore $\gamma(G_1)=\gamma(G_0)$.

\smallskip\noindent
\emph{Step 3: equality of $\eta$.}
For $\alpha\in A(G_0)$, and $k\in\{0,1\}$ define 
\begin{align*}
Z_{\alpha,k}(x)=\sum_{\pi\in \Sigma_k(\alpha)}x^{|\pi|}
\end{align*}
where $\Sigma_k(\alpha
)$ is the subset of $\Sigma_k$ containing SAWs ending at mid-edges in $\alpha X_k$.
Define $Z_{\alpha,1}^\ast(x)$ analogously but with SAWs on $G_1$ from mid-edges of $X_0$ to mid-edges of
$\alpha X_0$ (both viewed in $G_1$).  As in the proof of Theorem~3.3, endpoint-adjustment inside gadgets gives
a two-sided bound
\begin{equation}\label{eq:Zalpha-compare}
Z_{\alpha,1}^\ast(x)\;\le\; Z_{\alpha,1}(x)\;\le\;C(x)\,Z_{\alpha,1}^\ast(x),
\end{equation}
for a finite constant $C(x)$ depending only on the gadget and $x$.
Moreover, shrinking gadgets yields the exact identity
\begin{equation}\label{eq:Zalpha-subst}
Z_{\alpha,1}^\ast(x)\;=\;Z_{\alpha,0}(g(x)).
\end{equation}
Evaluating at $x=\mu_1^{-1}$ and using $\mu_0^{-1}=g(\mu_1^{-1})$ gives
$Z_{\alpha,1}^\ast(\mu_1^{-1})=Z_{\alpha,0}(\mu_0^{-1})$.

Finally, graph distances in $G_0$ and $G_1$ are comparable up to multiplicative constants:
there exists $\kappa\in(0,\infty)$ (depending only on the gadget) such that for all $\alpha$,
\[
d_{G_0}(W_0,\alpha W_0)\;\le\; d_{G_1}(W_1,\alpha W_1)\;\le\;\kappa\, d_{G_0}(W_0,\alpha W_0).
\]
Using \eqref{eq:Zalpha-compare}--\eqref{eq:Zalpha-subst} and the above distance comparison in the definition
\eqref{dd1} of $\eta$, we obtain $\eta(G_1)=\eta(G_0)$ (under the assumed existence of $\eta(G_1)$).
This completes the proof of part~(a).

\medskip\noindent
\textbf{(b) Equality of $\nu$ under \eqref{def:slow}.}
For $k=0,1$, let
\[
V_k(z)\;=\;\sum_{n=1}^\infty \frac{1}{n^{2z+1}}\,\Big\langle \|\pi\|_k^{\,2}\Big\rangle_{n,k},
\]
as in \eqref{def0}, and define
\begin{align}
V_1^\ast(z)\;:=\;\sum_{n=1}^\infty \frac{1}{n^{2z+1}}\,
\frac{1}{\sigma_{n,1}}\sum_{\pi\in\Sigma_{n,1}^\ast}\|\pi\|_1^{\,2}.\label{dv1s}
\end{align}
(Equivalently, $V_1^\ast(z)=\sum_{n\ge 1}n^{-(2z+1)}(\sigma_{n,1}^\ast/\sigma_{n,1})\langle\|\pi\|_1^2\rangle_{n,1}^\ast$.)

\smallskip\noindent
\emph{Step 4: $V_1(z)<\infty \iff V_1^\ast(z)<\infty$ for each finite $z$.}
Clearly $V_1^\ast(z)\le V_1(z)$.
For the reverse implication,

Recall that
\[
V_1(z)=\sum_{n\ge1}\frac{1}{n^{2z+1}}\Big\langle \|\pi\|_1^2\Big\rangle_{n,1}
=\sum_{n\ge1}\frac{1}{n^{2z+1}\sigma_{n,1}}\sum_{\pi\in\Sigma_{n,1}}\|\pi\|_1^2,
\]
and, by \eqref{dv1s},
\[
V_1^\ast(z)=\sum_{n\ge1}\frac{1}{n^{2z+1}}\frac{\sigma_{n,1}^\ast}{\sigma_{n,1}}
\Big\langle \|\pi\|_1^2\Big\rangle_{n,1}^\ast
=\sum_{n\ge1}\frac{1}{n^{2z+1}\sigma_{n,1}}\sum_{\pi\in\Sigma_{n,1}^\ast}\|\pi\|_1^2.
\]
Hence
\begin{equation}\label{eq:V1-decomp}
V_1(z)=V_1^\ast(z)+R(z),
\qquad
R(z):=\sum_{n\ge1}\frac{1}{n^{2z+1}\sigma_{n,1}}
\sum_{\pi\in\Sigma_{n,1}\setminus\Sigma_{n,1}^\ast}\|\pi\|_1^2.
\end{equation}

\smallskip\noindent
\textbf{(i) Walks contained in a single gadget give a finite additive constant.}
Let $\mathcal{G}$ be the set of SAWs in $\Sigma_{1}$ that are entirely contained in a single
gadget. Since the walk starts at a mid-edge in $X_1$, only finitely many gadgets intersect
the fundamental domain $W_1$, and each gadget is finite. Therefore $\mathcal{G}$ is finite,
and its total contribution
\[
C_z' := \sum_{\pi\in\mathcal{G}}
\frac{\|\pi\|_1^2}{|\pi|^{2z+1}\sigma_{|\pi|,1}}
\]
is finite for each fixed finite $z$.

\smallskip\noindent
\textbf{(ii) Trimming map to $\Sigma_{1}^\ast$ with bounded multiplicity.}
 For $\pi\in\Sigma_{n,1}\setminus\Sigma_{n,1}^\ast$ that is not
contained in a single gadget, at least one endpoint of $\pi$ is not a mid-edge of $E_0$.
Move each endpoint along $\pi$ to the first mid-edge of $E_0$ encountered, deleting the
initial segment at that endpoint. This produces a shorter SAW $\tau(\pi)\in\Sigma_{m,1}^\ast$
with
\[
m=|\tau(\pi)|\in\{n-j:\ 0\le j\le 2N\}.
\]
Moreover, 
\[
K:=\Bigl(1+\sum_{i=1}^{N}D^{i}\Bigr)^2<\infty
\]
such that every $\pi^\ast\in\Sigma_{m,1}^\ast$ has at most $K$ preimages (each endpoint can be
extended by at most $N$ steps, each step has at most $D$ choices; there are two endpoints) under $\tau$.

Also, trimming changes each endpoint by at most $N$ vertices within a gadget, hence changes
the endpoint distance by at most an additive constant. In particular,
\[
\|\pi\|_1 \le \|\tau(\pi)\|_1 + 2M.
\]
Since $|\tau(\pi)|\ge 1$ implies $\|\tau(\pi)\|_1\ge 1$, we have
\begin{equation}\label{eq:dist-bound}
\|\pi\|_1^2 \le (1+2M)^2\,\|\tau(\pi)\|_1^2 .
\end{equation}

\smallskip\noindent
\textbf{(iii) Comparing the normalizing factors using \eqref{def:slow}.}
Fix $j\in\{0,\dots,2N\}$ and finite $z$. For $m\ge1$,
\begin{equation}\label{eq:weight-compare}
\frac{1}{(m+j)^{2z+1}\sigma_{m+j,1}}
=
\Bigl(\frac{m}{m+j}\Bigr)^{2z+1}\frac{\sigma_{m,1}}{\sigma_{m+j,1}}
\cdot \frac{1}{m^{2z+1}\sigma_{m,1}}.
\end{equation}
By \eqref{def:slow} (i.e.\ (3.7)) for $k=1$,
\[
\sigma_{n,1}\asymp L(n)\,n^{\gamma-1}\mu_1^{\,n},
\]
so for fixed $j$,
\[
\frac{\sigma_{m,1}}{\sigma_{m+j,1}}
\le
\frac{C_{2,1}}{C_{1,1}}\,
\frac{L(m)}{L(m+j)}\,
\Bigl(\frac{m}{m+j}\Bigr)^{\gamma-1}\,
\mu_1^{-j}.
\]
Since $L$ is slowly varying, $L(m)/L(m+j)$ is bounded in $m$ for each fixed $j$,
and $\frac{m}{m+j}\in[1/(1+j),1]$ so the power term is bounded as well. Therefore the product
\[
A_{z,j}:=\sup_{m\ge1}\Bigl(\frac{m}{m+j}\Bigr)^{2z+1}\frac{\sigma_{m,1}}{\sigma_{m+j,1}}<\infty,
\]
and \eqref{eq:weight-compare} yields
\begin{equation}\label{eq:factor-bound}
\frac{1}{(m+j)^{2z+1}\sigma_{m+j,1}}
\le A_{z,j}\cdot \frac{1}{m^{2z+1}\sigma_{m,1}}.
\end{equation}

\smallskip\noindent
\textbf{(iv) Bounding $R(z)$ by $V_1^\ast(z)$.}
Using \eqref{eq:dist-bound}, bounded multiplicity $K$, and \eqref{eq:factor-bound}, we obtain
\begin{align*}
R(z)
&\le C_z' + \sum_{j=0}^{2N}\sum_{m\ge1}
\frac{1}{(m+j)^{2z+1}\sigma_{m+j,1}}
\sum_{\substack{\pi\in\Sigma_{m+j,1}\setminus\Sigma_{m+j,1}^\ast\\ |\pi|-|\tau(\pi)|=j}}
\|\pi\|_1^2 \\
&\le C_z' + \sum_{j=0}^{2N}\sum_{m\ge1}
\frac{1}{(m+j)^{2z+1}\sigma_{m+j,1}}
\sum_{\pi^\ast\in\Sigma_{m,1}^\ast} K\,(1+2N)^2 \|\pi^\ast\|_1^2 \\
&\le C_z' + \sum_{j=0}^{2N} K\,(1+2N)^2 A_{z,j}
\sum_{m\ge1}\frac{1}{m^{2z+1}\sigma_{m,1}}\sum_{\pi^\ast\in\Sigma_{m,1}^\ast}\|\pi^\ast\|_1^2 \\
&= C_z' + C_z\, V_1^\ast(z),
\end{align*}
where $C_z:=\sum_{j=0}^{2N} K(1+2N)^2A_{z,j}<\infty$.
Combining this with \eqref{eq:V1-decomp} gives the explicit bound
\begin{equation}\label{eq:V1-upper-explicit}
V_1(z)\le (1+C_z)\,V_1^\ast(z)+C_z'.
\end{equation}
Since also $V_1^\ast(z)\le V_1(z)$, we conclude that for each fixed finite $z$,
\[
V_1(z)<\infty \iff V_1^\ast(z)<\infty,
\]
which is \eqref{eq:V1-V1star}.

Therefore \begin{equation}\label{eq:V1-V1star} V_1(z)\;<\infty \iff V_1^\ast(z)\;<\infty .
\end{equation}

\smallskip\noindent
\emph{Step 5: $V_1^\ast(z)<\infty \iff V_0(z)<\infty$ for each finite $z$.}
Fix $n\ge 1$ and $\omega\in\Sigma_{n,0}$.  Let $\mathcal{E}(\omega)$ be the family of SAWs
$\pi\in\Sigma_1^\ast$ that contract to $\omega$.  As above, for each $m$ the number of $\pi\in\mathcal{E}(\omega)$
with $|\pi|=m$ equals $b_{n,m}$ (the coefficient of $x^m$ in $g(x)^n$), and $b_{n,m}=0$ unless $2n\le m\le Nn$.

Since endpoints of $\pi$ and $\omega$ correspond to the same two original mid-edges, the graph distances satisfy
$\|\omega\|_0 \le \|\pi\|_1 \le \kappa \|\omega\|_0$ for some $\kappa<\infty$ depending only on the gadget.
Hence, for fixed $\omega$,
\[
\sum_{\pi\in\mathcal{E}(\omega)} \frac{\|\pi\|_1^{\,2}}{|\pi|^{2z+1}\,\sigma_{|\pi|,1}}
\asymp
\|\omega\|_0^{\,2}\sum_{m=2n}^{Nn} b_{n,m}\,\frac{1}{m^{2z+1}\,\sigma_{m,1}},
\]
where $\asymp$ hides multiplicative constants depending only on $(z,\kappa)$.

Now apply \eqref{def:slow} (for $k=1$): for all $m\ge 1$,
\[
\frac{1}{\sigma_{m,1}}
\asymp
\frac{\mu_1^{-m}}{L(m)\,m^{\gamma-1}}.
\]
Therefore,
\[
\sum_{m=2n}^{Nn} b_{n,m}\,\frac{1}{m^{2z+1}\,\sigma_{m,1}}
\asymp
\sum_{m=2n}^{Nn} b_{n,m}\,\frac{\mu_1^{-m}}{L(m)\,m^{2z+\gamma}}.
\]
Since $L$ is slowly varying and $m\in[2n,Nn]$, the values $L(m)$ are comparable to $L(n)$ up to a constant
(depending on $N$).  Also $m^{-(2z+\gamma)}$ is comparable to $n^{-(2z+\gamma)}$ up to a constant depending
on $(z,N)$.  Hence
\[
\sum_{m=2n}^{Nn} b_{n,m}\,\frac{\mu_1^{-m}}{L(m)\,m^{2z+\gamma}}
\asymp
\frac{1}{L(n)\,n^{2z+\gamma}}\sum_{m=2n}^{Nn} b_{n,m}\,\mu_1^{-m}.
\]
But $\sum_m b_{n,m}\mu_1^{-m}=g(\mu_1^{-1})^n=\mu_0^{-n}$ by Theorem~3.3.
Thus,
\[
\sum_{\pi\in\mathcal{E}(\omega)} \frac{\|\pi\|_1^{\,2}}{|\pi|^{2z+1}\,\sigma_{|\pi|,1}}
\asymp
\|\omega\|_0^{\,2}\,\frac{\mu_0^{-n}}{L(n)\,n^{2z+\gamma}}.
\]
Summing over $\omega\in\Sigma_{n,0}$ and then over $n\ge 1$ yields
\[
V_1^\ast(z)\asymp \sum_{n\ge 1}\frac{1}{\sigma_{n,0}}\frac{1}{n^{2z+1}}\sum_{\omega\in\Sigma_{n,0}}\|\omega\|_0^{\,2}
\;=\;V_0(z),
\]
where we used again \eqref{def:slow} for $k=0$ to identify
$\sigma_{n,0}^{-1}n^{-(2z+1)}\asymp \mu_0^{-n}(L(n)n^{2z+\gamma})^{-1}$.
Consequently, for finite $z$,
\begin{equation}\label{eq:V0-V1star}
V_1^\ast(z)<\infty \iff V_0(z)<\infty .
\end{equation}

Combining \eqref{eq:V1-V1star} and \eqref{eq:V0-V1star}, we conclude that
$V_1(z)<\infty \iff V_0(z)<\infty$ for all finite $z$, hence $\nu(G_1)=\nu(G_0)$.
\end{proof}

\section{Proof of Theorem \ref{thm:main3}}\label{sec:proof3}

\begin{proof}[Proof of Theorem~3.5]
We prove part~(2). Part~(1) follows by taking $\phi_2$ to be the trivial
transformation (so that $g_{\phi_2}(x)=x$ in the mid-edge convention), and by
noting that the argument below never uses that white vertices have degree $3$
except when defining $g_{\phi_2}$.

\medskip
\noindent\textbf{Setup.}
Let $G=(V,E)$ be an infinite, connected, bipartite graph with color classes
$V_{\mathrm{black}}$ and $V_{\mathrm{white}}$, and suppose $G$ is quasi-transitive as a colored graph.
Assume every vertex has degree $3$.
Let $G_{\mathrm e}$ be obtained from $G$ by applying the local transformation
$\phi_1$ at each black vertex and $\phi_2$ at each white vertex.
Let $W\subset V$ be a minimal fundamental domain for the colored automorphism group,
and let $X\subset E$ be the set of edges incident to vertices of $W$.
Let $W_{\mathrm e}$ be the set of vertices of $G_{\mathrm e}$ lying in gadgets
over vertices of $W$, and let $X_{\mathrm e}$ be the set of edges of $G_{\mathrm e}$
incident to vertices in $W_{\mathrm e}$.
As before, we work with SAWs starting/ending at mid-edges, and the mid-edges of $E$
(resp.\ $X$) are viewed as mid-edges of $E_{\mathrm e}$ (resp.\ $X_{\mathrm e}$).

For $n\ge0$, let $s_n$ be the number of $n$-step SAWs on $G_{\mathrm e}$ starting at
mid-edges in $X_{\mathrm e}$.
Let $c_n$ be the number of $n$-step SAWs on $G_{\mathrm e}$ starting at mid-edges in $X$
and ending at mid-edges in $E$ (both viewed inside $G_{\mathrm e}$).
Clearly,
\begin{equation}\label{eq:cn-le-sn}
c_n \le s_n \qquad (n\ge0).
\end{equation}

\medskip
\noindent\textbf{Step 1: $s_n$ and $c_n$ have the same exponential growth rate.}
Let
\begin{align*}
&N:=\max\{|V(\phi_1(v))|,\ |V(\phi_2(v))|\},\\
&D:=\max\{\deg(u): u\text{ lies in a gadget of }\phi_1\text{ or }\phi_2\}.
\end{align*}
Every $n$-step SAW counted by $s_n$ falls into one of two classes:
\begin{itemize}
\item it is entirely contained in a single gadget (a finite graph);
\item otherwise, it visits at least one mid-edge belonging to $E$.
\end{itemize}
In the first case $n\leq N$.
In the second case, trim each endpoint that is not a mid-edge of $E$ along the walk
to the first encountered mid-edge of $E$. This removes at most $N$ visited vertices
per endpoint, hence at most $2N$ in total, producing a shorter SAW counted by
$c_{n-i}$ for some $i\in\{0,1,\dots,2N\}$.

Conversely, given a walk counted by $c_{n-i}$, the number of ways to extend its
endpoints by $i$ additional vertices (inside gadgets) is bounded by $(i+1)D^i$.
Therefore there exists a constant $C_0<\infty$ (depending only on the gadgets and $W$)
such that, with the convention $c_m=0$ for $m<0$,
\begin{equation}\label{eq:sn-bound}
s_n \le \sum_{i=0}^{2N} (i+1)D^i\,c_{n-i}\qquad (n\ge N+1).
\end{equation}
Combining \eqref{eq:cn-le-sn} and \eqref{eq:sn-bound} yields
\[
\limsup_{n\to\infty} s_n^{1/n}=\limsup_{n\to\infty} c_n^{1/n}.
\]
Since $G_{\mathrm e}$ is quasi-transitive, the connective constant $\mu_{\mathrm e}$ exists and
equals this common growth rate. In particular, the power series
\[
C(x):=\sum_{n\ge0} c_n x^n
\]
has radius of convergence $\mu_{\mathrm e}^{-1}$.

\medskip
\noindent\textbf{Step 2: even SAWs on $G$ and the bivariate partition function.}
Let $\mathcal E$ be the set of SAWs on $G$ starting at mid-edges in $X$,
and let $\mathcal E_{\mathrm{even}}\subset\mathcal E$ be those of even length.
For $\pi\in\mathcal E$, write $|\pi_{\mathrm{black}}|$ (resp.\ $|\pi_{\mathrm{white}}|$) for the number
of black (resp.\ white) vertices visited by $\pi$; then $|\pi|=|\pi_{\mathrm{black}}|+|\pi_{\mathrm{white}}|$.

For $x,y>0$, define
\[
Z(x,y):=\sum_{\pi\in\mathcal E} x^{|\pi_{\mathrm{black}}|}y^{|\pi_{\mathrm{white}}|},
\qquad
Z_{\mathrm{even}}(x,y):=\sum_{\pi\in\mathcal E_{\mathrm{even}}} x^{|\pi_{\mathrm{black}}|}y^{|\pi_{\mathrm{white}}|}.
\]
Clearly $Z_{\mathrm{even}}(x,y)\le Z(x,y)$.
Let $\Delta:=\sup_{v\in V}\deg(v)<\infty$ (here $\Delta=3$).
Every odd-length SAW in $\mathcal E$ is obtained from an even-length SAW by appending
one extra step at one endpoint; the number of choices is bounded by $\Delta$, and the
one-step SAWs contribute a finite constant depending on $|X|$.
Consequently, for all $x,y>0$,
\begin{equation}\label{eq:Z-even-compare}
Z(x,y)-Z_{\mathrm{even}}(x,y)\ \le\ (\Delta x+\Delta y)\,\bigl(|X|+Z_{\mathrm{even}}(x,y)\bigr),
\end{equation}
and hence
\begin{equation}\label{eq:Z-finite-iff}
Z(x,y)<\infty \iff Z_{\mathrm{even}}(x,y)<\infty.
\end{equation}

Since $G$ is bipartite, any even-length SAW alternates colors and satisfies
$|\pi_{\mathrm{black}}|=|\pi_{\mathrm{white}}|=|\pi|/2$, and therefore
\[
Z_{\mathrm{even}}(x,y)=\sum_{\pi\in\mathcal E_{\mathrm{even}}} (xy)^{|\pi|/2}
=\sum_{m\ge0} \sigma_{2m}\,(xy)^m,
\]
where $\sigma_{n}$ is the number of $n$-step SAWs in $\mathcal E$.
Thus the radius of convergence of $Z_{\mathrm{even}}(x,y)$ as a function of $t:=xy$
equals $\mu^{-2}$, where $\mu=\mu(G)$ is the connective constant of $G$.
In particular,
\begin{equation}\label{eq:Z-critical}
Z(x,y)<\infty \iff xy<\mu^{-2}.
\end{equation}

\medskip
\noindent\textbf{Step 3: substitution identity and the connective-constant equation.}
Let
\[
g_1(x):=g_{\phi_1}(x),\qquad g_2(x):=g_{\phi_2}(x),
\]
where each $g_i$ is defined as in \eqref{dgx} (Definition~3.2) for the corresponding gadget.
Consider SAWs counted by $c_n$, i.e.\ SAWs in $G_{\mathrm e}$ that start at a mid-edge of $X$
and end at a mid-edge of $E$.
Contracting each gadget of $G_{\mathrm e}$ to its original vertex yields a map from such SAWs
to SAWs $\pi\in\mathcal E$ on $G$.
Conversely, given $\pi\in\mathcal E$, each time $\pi$ visits a black vertex it uses two
incident edges, and the portion of the lifted walk inside the corresponding black gadget
is any SAW connecting the two corresponding incident mid-edges; the total contribution of
all choices is exactly $g_1(x)$.
Similarly each visited white vertex contributes a factor $g_2(x)$.
Since $\pi$ is self-avoiding, these gadget-choices occur in disjoint gadgets and hence multiply.
Therefore we obtain the exact identity
\begin{equation}\label{eq:C-subst}
C(x)=\sum_{n\ge0}c_n x^n \;=\;\sum_{\pi\in\mathcal E} g_1(x)^{|\pi_{\mathrm{black}}|}\,g_2(x)^{|\pi_{\mathrm{white}}|}
\;=\; Z\bigl(g_1(x),g_2(x)\bigr).
\end{equation}
By \eqref{eq:Z-critical}, we have
\[
C(x)<\infty \iff g_1(x)\,g_2(x)<\mu^{-2}.
\]
Hence the radius of convergence of $C(x)$ is the unique positive solution $\rho$ of
\[
g_1(\rho)\,g_2(\rho)=\mu^{-2}.
\]
Since the radius of $C(x)$ is $\mu_{\mathrm e}^{-1}$ by Step~1, this yields
\[
\mu^{-2}=g_{\phi_1}(\mu_{\mathrm e}^{-1})\,g_{\phi_2}(\mu_{\mathrm e}^{-1}),
\]
which is Theorem~3.5(2)(a).

For part~(1)(a), take $\phi_2$ to be the trivial transformation; in the mid-edge convention,
traversing an untransformed vertex contributes one visited vertex, hence $g_{\phi_2}(x)=x$,
and the identity above becomes $\mu^{-2}=x\,g_{\phi_1}(x)$ with $x=\mu_{\mathrm e}^{-1}$.

\medskip
\noindent\textbf{Step 4: invariance of critical exponents (parts (1)(b) and (2)(b)).}
Under the corresponding assumptions of Theorem~3.4, one repeats the arguments of
Section~4 with the contraction/expansion map used in \eqref{eq:C-subst} (and the basic
even/odd comparison \eqref{eq:Z-finite-iff}) to compare the finiteness of the analogues of
$Y(\mu^{-1},y)$, $Z_\alpha(\mu^{-1})$, and $V(z)$ for $G$ and $G_{\mathrm e}$.
The gadget lengths are uniformly bounded above and below, so the polynomial weights in
$|\pi|$ are affected only by multiplicative constants as in the proof of Theorem~3.4.
This implies that the exponents $\gamma,\eta$ are preserved, and that $\nu$ is preserved
under the additional hypothesis \eqref{def:slow} with a common slowly varying function.
\end{proof}

\section{Examples}\label{sec:examples}

This section collects explicit gadget families and worked base-graph examples illustrating
Theorems~\ref{thm:main2} and~\ref{thm:main3}.  Many natural quasi-transitive cubic graphs
are planar (e.g.\ the hexagonal lattice), while others are not (e.g.\ the $3$-regular tree);
our substitution identities apply uniformly in all cases.

\begin{example}[Complete-graph gadgets $K_N$]\label{ex:KN}
Fix an integer $N\ge 3$. Define a local transformation $\phi^{(N)}$ as follows:
each degree-$3$ vertex $v$ is replaced by a copy of the complete graph $K_N$ with three
distinguished (outer) vertices $w_1,w_2,w_3$, and the three edges incident to $v$ are
attached to $w_1,w_2,w_3$ (one per edge).
Since $\mathrm{Aut}(K_N)\cong S_N$ contains permutations acting transitively on
$\{w_1,w_2,w_3\}$ while preserving this set, $\phi^{(N)}$ satisfies Definition~\ref{df31}.

Let $g_N(x)=g_{\phi^{(N)}}(x)$ be as in \eqref{dgx}.  Consider SAWs in $K_N$ that start at the
mid-edge of the edge attached to $w_1$ and end at the mid-edge attached to $w_2$.
In the mid-edge convention, such a walk visits a sequence of vertices
\[
w_1=v_1, v_2,\dots,v_m=w_2,
\]
with all $v_i$ distinct. Since $K_N$ is complete, every such sequence is a SAW. Hence, for
each $m\in\{2,3,\dots,N\}$ the number of $m$-step gadget SAWs equals the number of ordered
choices of the $(m-2)$ internal vertices from the $(N-2)$ available vertices
$V(K_N)\setminus\{w_1,w_2\}$, namely
\[
a_m = (N-2)(N-3)\cdots (N-m+1) = \frac{(N-2)!}{(N-m)!}.
\]
Therefore
\begin{equation}\label{eq:gN}
g_N(x)=\sum_{m=2}^{N} a_m x^m
=\sum_{m=2}^{N} \frac{(N-2)!}{(N-m)!}\,x^m
= x^2\sum_{k=0}^{N-2} (N-2)_k\, x^k,
\end{equation}
where $(N-2)_k=(N-2)(N-3)\cdots (N-1-k)$ is the falling factorial.

By Theorem~\ref{thm:main2}, if $G$ is cubic with connective constant $\mu$ and
$G_1=\phi^{(N)}(G)$ has connective constant $\mu_1$, then
\[
\mu^{-1}=g_N(\mu_1^{-1}),
\]
and $\mu_1^{-1}$ is the unique solution $x\in(0,1)$ of $g_N(x)=\mu^{-1}$.

See Figure~\ref{fig:K7gadget} for an example of a $K_N$ gadget with $N=7$.
\end{example}

\begin{remark}[A worked base graph: the $3$-regular tree]\label{rmk:T3}
Let $T_3$ denote the $3$-regular tree. Then $\mu(T_3)=2$ (indeed $\sigma_n=3\cdot 2^{n-1}$).
For $G_1=\phi^{(N)}(T_3)$, Theorem~\ref{thm:main2} yields $g_N(\mu_1^{-1})=\tfrac12$.
For instance,
\[
\mu_1 \approx 2.16730\ (N=4),\qquad
\mu_1 \approx 2.60068\ (N=5),\qquad
\mu_1 \approx 3.55052\ (N=7),
\]
obtained by solving $g_N(x)=\tfrac12$ for $x\in(0,1)$ and setting $\mu_1=1/x$.
This illustrates how complete-graph gadgets produce explicit, finite polynomial equations
for new connective constants.
\end{remark}

\begin{figure}[t]
\centering
\begin{tikzpicture}[scale=1.3, line cap=round, line join=round]
  \coordinate (Aout) at (-1.35,0.75);
  \coordinate (Bout) at ( 1.35,0.75);
  \coordinate (Cout) at ( 0,-1.45);

  \coordinate (a) at (-0.825,0.525);
  \coordinate (b) at ( 0.825,0.525);
  \coordinate (c) at ( 0,-0.975);
  \coordinate (dNE) at (0.3,0.4);
  \coordinate (dSE) at (0.3,-0.2);
  \coordinate (dSW) at (-0.3,-0.2);
  \coordinate (dNW) at (-0.3,0.4);

  \draw (a)--(b)--(c)--cycle;
  \draw (dNE)--(dSE)--(dSW)--(dNW)--cycle;
  \draw (dNE)--(dSW);
  \draw (dNW)--(dSE);
  \draw (dNE)--(a) (dNE)--(b) (dNE)--(c);
  \draw (dSE)--(a) (dSE)--(b) (dSE)--(c);
  \draw (dSW)--(a) (dSW)--(b) (dSW)--(c);
  \draw (dNW)--(a) (dNW)--(b) (dNW)--(c);
  \draw (a)--(Aout);
  \draw (b)--(Bout);
  \draw (c)--(Cout);

  \fill (a) circle (1.1pt);
  \fill (b) circle (1.1pt);
  \fill (c) circle (1.1pt);
  \fill (dNE) circle (1.1pt);
  \fill (dSE) circle (1.1pt);
  \fill (dSW) circle (1.1pt);
  \fill (dNW) circle (1.1pt);

  \node[above left]  at (Aout) {\small $A$};
  \node[above right] at (Bout) {\small $B$};
  \node[below]       at (Cout) {\small $C$};

  \node[above] at (0,0.95) {\small $K_7$ gadget};
\end{tikzpicture}
\caption{A local transformation consisting of a $K_7$ gadget.
In this case, $g_7(x)=x^2+5x^3+20x^4+60x^5+120x^6+120x^7$.}
\label{fig:K7gadget}
\end{figure}
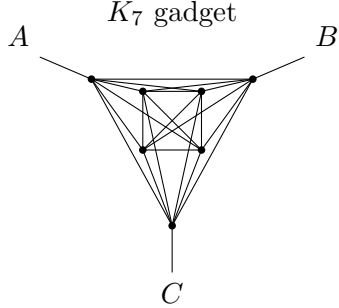

\begin{example}[A planar six-vertex ``triangle-in-triangle'' gadget]\label{ex:tri6}
Let $\phi^{\triangle}$ be the local transformation whose gadget $T$ is the following
$6$-vertex graph (see Figure~\ref{fig:tri6gadget}).
The gadget has three \emph{ports} (outer vertices) $w_1,w_2,w_3$.
On each side $w_iw_j$ of the outer triangle we insert one vertex $u_{ij}$, so the boundary
edges are
\[
w_1-u_{12}-w_2,\qquad w_2-u_{23}-w_3,\qquad w_3-u_{31}-w_1,
\]
and we add the three edges forming the inner triangle
\[
u_{12}u_{23},\qquad u_{23}u_{31},\qquad u_{31}u_{12}.
\]
The dihedral symmetries act transitively on $\{w_1,w_2,w_3\}$, so $\phi^{\triangle}$
satisfies Definition~\ref{df31}.

Let $g_{\triangle}(x):=g_{\phi^{\triangle}}(x)$ be defined by \eqref{dgx}.  A gadget-SAW
from the port-edge at $w_1$ to that at $w_2$ is exactly a simple path from $w_1$ to $w_2$
inside $T$ (mid-edge convention).  Enumerating simple paths by their number of visited
vertices gives
\[
g_{\triangle}(x)=x^3+3x^4+4x^5+2x^6.
\]
Therefore, for any infinite connected quasi-transitive cubic graph $G$ with connective
constant $\mu$ and $G_1=\phi^{\triangle}(G)$ with connective constant $\mu_1$, Theorem~\ref{thm:main2}
yields the exact relation
\[
\mu^{-1}=g_{\triangle}\!\bigl(\mu_1^{-1}\bigr).
\]

\smallskip
\noindent\emph{Hexagonal lattice.}
If $G=\mathbb H$ is the hexagonal lattice, then $\mu(\mathbb H)=\sqrt{2+\sqrt 2}$ \cite{ds}, so
$\mu(\phi^{\triangle}(\mathbb H))^{-1}\in(0,1)$ is the unique solution of
\[
x^3+3x^4+4x^5+2x^6=\bigl(2+\sqrt 2\bigr)^{-1/2}.
\]
(Numerically, this gives $\mu(\phi^{\triangle}(\mathbb H))\approx 1.9318771$.)
\end{example}

\begin{figure}[t]
\centering
\begin{tikzpicture}[scale=1.25, line cap=round, line join=round]
  \coordinate (W1) at (-1.6,0.8);
  \coordinate (W2) at ( 1.6,0.8);
  \coordinate (W3) at ( 0.0,-1.55);

  \coordinate (U12) at ( 0.0,0.8);
  \coordinate (U23) at ( 0.8,-0.375);
  \coordinate (U31) at (-0.8,-0.375);

  \coordinate (Aout) at (-2.35,1.20);
  \coordinate (Bout) at ( 2.35,1.20);
  \coordinate (Cout) at ( 0.0,-2.35);

  \draw (W1)--(U12)--(W2);
  \draw (W2)--(U23)--(W3);
  \draw (W3)--(U31)--(W1);

  \draw (U12)--(U23)--(U31)--cycle;

  \draw (W1)--(Aout);
  \draw (W2)--(Bout);
  \draw (W3)--(Cout);

  \fill (W1) circle (1.25pt);
  \fill (W2) circle (1.25pt);
  \fill (W3) circle (1.25pt);
  \fill (U12) circle (1.25pt);
  \fill (U23) circle (1.25pt);
  \fill (U31) circle (1.25pt);

  \node[left]  at (Aout) {\small $A$};
  \node[right] at (Bout) {\small $B$};
  \node[below] at (Cout) {\small $C$};

  \node[above] at (0,1.25) {\small triangle-in-triangle gadget};
\end{tikzpicture}
\caption{A planar $6$-vertex gadget with outer ports $w_1,w_2,w_3$ and an inner triangle supported on the three sides. For this gadget,
$g_{\triangle}(x)=x^3+3x^4+4x^5+2x^6$.}
\label{fig:tri6gadget}
\end{figure}
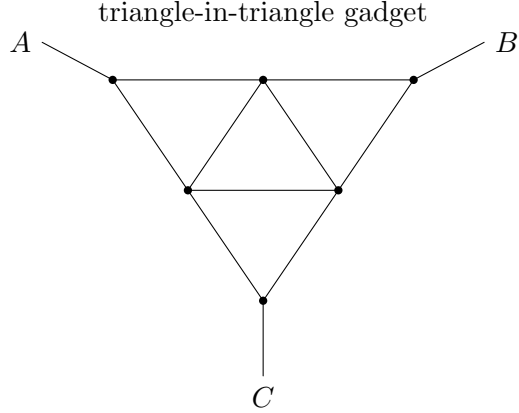

\begin{example}[A mixed bipartite transformation on the hexagonal lattice]\label{ex:hex-mixed}
Let $H$ be the hexagonal lattice with its natural bipartition $V_{\mathrm w}\cup V_{\mathrm b}$
(white/black).  Apply $\phi^{(4)}$ at each white vertex (a $K_4$ gadget), and apply the Fisher
triangle transformation at each black vertex (i.e.\ a $K_3$ gadget).  Denote the resulting
graph by $H_{\mathrm e}$.

For the Fisher triangle one has
\[
g_{K_3}(x)=x^2+x^3,
\]
and for the $K_4$ gadget one has $g_{K_4}(x)=x^2+2x^3+2x^4$.
By Theorem~\ref{thm:main3}(2),
\[
\mu(H)^{-2}=g_{K_4}\bigl(\mu(H_{\mathrm e})^{-1}\bigr)\,g_{K_3}\bigl(\mu(H_{\mathrm e})^{-1}\bigr).
\]
Since $\mu(H)^{-2}=(2+\sqrt2)^{-1}=\tfrac12(2-\sqrt2)$, writing $x:=\mu(H_{\mathrm e})^{-1}$ yields
the explicit polynomial equation
\begin{equation}\label{eq:hex-mixed-eqn}
x^4+3x^5+4x^6+2x^7=\tfrac12(2-\sqrt2),
\qquad x\in(0,1).
\end{equation}
Thus $\mu(H_{\mathrm e})=1/x \approx 1.91529$.

\smallskip
The two gadgets used here are depicted in Figure~\ref{fig:K3K4}.
\end{example}

\begin{figure}[t]
\centering
\setlength{\tabcolsep}{28pt}
\renewcommand{\arraystretch}{1.0}
\begin{tabular}{cc}

\begin{tikzpicture}[scale=1.15, line cap=round, line join=round]
  \coordinate (Aout) at (-1.85,1.10);
  \coordinate (Bout) at ( 1.85,1.10);
  \coordinate (Cout) at ( 0,-2.05);

  \coordinate (a) at (-0.95,0.55);
  \coordinate (b) at ( 0.95,0.55);
  \coordinate (c) at ( 0,-1.10);

  \draw (a)--(b)--(c)--cycle;
  \draw (a)--(Aout);
  \draw (b)--(Bout);
  \draw (c)--(Cout);

  \fill (a) circle (1.4pt);
  \fill (b) circle (1.4pt);
  \fill (c) circle (1.4pt);

  \node[above left]  at (Aout) {\small $A$};
  \node[above right] at (Bout) {\small $B$};
  \node[below]       at (Cout) {\small $C$};

  \node[above] at (0,0.95) {\small $K_3$ (Fisher triangle)};
\end{tikzpicture}

&

\begin{tikzpicture}[scale=1.15, line cap=round, line join=round]
  \coordinate (Aout) at (-1.85,1.10);
  \coordinate (Bout) at ( 1.85,1.10);
  \coordinate (Cout) at ( 0,-2.05);

  \coordinate (a) at (-0.95,0.55);
  \coordinate (b) at ( 0.95,0.55);
  \coordinate (c) at ( 0,-1.10);
  \coordinate (d) at ( 0,-0.05);

  \draw (a)--(b)--(c)--cycle; 
  \draw (d)--(a) (d)--(b) (d)--(c); 

  \draw (a)--(Aout);
  \draw (b)--(Bout);
  \draw (c)--(Cout);

  \fill (a) circle (1.4pt);
  \fill (b) circle (1.4pt);
  \fill (c) circle (1.4pt);
  \fill (d) circle (1.4pt);

  \node[above left]  at (Aout) {\small $A$};
  \node[above right] at (Bout) {\small $B$};
  \node[below]       at (Cout) {\small $C$};

  \node[above] at (0,0.95) {\small $K_4$ gadget};
\end{tikzpicture}

\end{tabular}
\caption{The two gadgets used in Example~\ref{ex:hex-mixed}: a Fisher triangle ($K_3$) and a $K_4$ gadget with ports $A,B,C$.}
\label{fig:K3K4}
\end{figure}
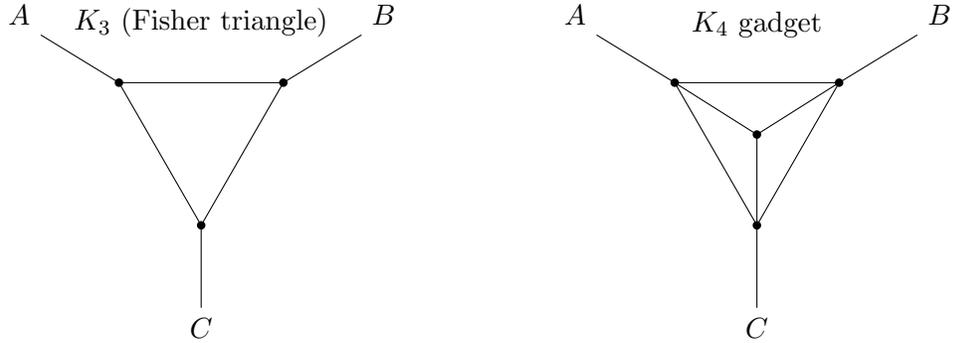

\begin{example}[Generalized Fisher transformations]\label{ex:GFT}
Let $H$ be a finite connected graph with two distinct distinguished vertices $v_1,v_2$.
Let $H^{(1)},H^{(2)},H^{(3)}$ be three disjoint copies of $H$, and write
$v_{1}^{(i)},v_{2}^{(i)}$ for the copies of $v_1,v_2$ in $H^{(i)}$.
Form a graph $\widetilde H$ by identifying
\[
v_{2}^{(i)} \equiv v_{1}^{(i+1)} \qquad (i \in \mathbb Z/3\mathbb Z).
\]
Let the three ports (outer vertices) be $w_i:=v_{1}^{(i)}$ for $i=1,2,3$.
Then $\widetilde H$ is a gadget: the cyclic permutation of the three copies induces
automorphisms acting transitively on $\{w_1,w_2,w_3\}$.
Replacing every degree-$3$ vertex of a cubic graph $G$ by $\widetilde H$ (attaching the three
incident edges to $w_1,w_2,w_3$) defines a local transformation, denoted $\phi^H$.
When $H$ is a single edge, $\phi^H$ is the usual Fisher transformation.

To describe the corresponding series $g_H(x):=g_{\phi^H}(x)$, let $H'$ be obtained from $H$
by attaching one extra leaf-edge at each of $v_1$ and $v_2$, and let $e_1,e_2$ be these two new edges.
Define the two-terminal SAW series
\[
f_H(x):=\sum_{\pi\in\Sigma_{H',e_1,e_2}} x^{|\pi|}.
\]
Then
\begin{equation}\label{eq:gH}
g_H(x)=f_H(x)+x^{-1}\,f_H(x)^2.
\end{equation}
Indeed, a gadget-SAW from the port attached to $w_1$ to that attached to $w_2$
either avoids $w_3$ (hence lies in a single copy of $H$ and contributes $f_H$),
or visits $w_3$ (hence splits into two such walks, with $w_3$ counted twice, giving the factor $x^{-1}$).

Consequently, if $G$ has connective constant $\mu$ and $\phi^H(G)$ has connective constant $\mu_H$,
then by Theorem~\ref{thm:main2},
\[
\mu^{-1}=g_H(\mu_H^{-1}),
\qquad\text{and }\ \mu_H^{-1}\in(0,1)\ \text{is the unique solution of }\ g_H(x)=\mu^{-1}.
\]

\smallskip
\noindent\emph{Subexamples.}
\begin{enumerate}
\item If $H$ is the path on $N\ge2$ vertices with endpoints $v_1,v_2$, then $f_H(x)=x^N$ and
\[
g_H(x)=x^N+x^{2N-1}.
\]
\item If $H$ is the cycle $C_N$ ($N\ge3$) and the graph-distance between $v_1,v_2$ along the shorter arc is $k$,
then
\[
f_H(x)=x^{k+1}+x^{N-k+1},
\qquad
g_H(x)=f_H(x)+x^{-1}f_H(x)^2.
\]
\end{enumerate}
\end{example}

\begin{example}[Compositions of local transformations]\label{ex:compositions}
Let $G$ be a cubic graph. Suppose $\phi_1$ is a local transformation such that
$G_1:=\phi_1(G)$ is again cubic, and let $\phi_2$ be a local transformation on $G_1$ with
$G_2:=\phi_2(G_1)$.
Write $g_i(x):=g_{\phi_i}(x)$ and $\mu_i:=\mu(G_i)$.
By Theorem~\ref{thm:main2},
\[
\mu_0^{-1}=g_1(\mu_1^{-1}),\qquad \mu_1^{-1}=g_2(\mu_2^{-1}),
\]
hence
\[
\mu_0^{-1}=(g_1\circ g_2)(\mu_2^{-1}).
\]
Equivalently, $\mu_2^{-1}\in(0,1)$ is the unique solution of $(g_1\circ g_2)(x)=\mu_0^{-1}$.

More generally, for a sequence $\phi_1,\dots,\phi_r$ such that $G_i:=\phi_i(G_{i-1})$ is cubic for
$1\le i\le r$, one has
\[
\mu_0^{-1}=(g_1\circ\cdots\circ g_r)(\mu_r^{-1}),
\qquad
g_{\phi_r\circ\cdots\circ \phi_1}=g_1\circ\cdots\circ g_r.
\]
\end{example}

\begin{example}[A ``hub--triangle'' gadget and an explicit iteration picture]\label{ex:hub-triangle}
Define a local transformation $\psi$ by replacing each degree-$3$ vertex by the gadget obtained as follows:
start from a triangle on outer vertices $w_1,w_2,w_3$ (the ports), subdivide each of the three triangle edges by
inserting a new vertex at its midpoint, and finally add a central vertex connected to the three midpoint vertices.
(See Figure~\ref{fig:triHubFixedPoint}.)  This gadget is port-transitive by the dihedral symmetry of the construction.

A direct enumeration of gadget SAWs (in the mid-edge convention of Section~\ref{sec:notation}) yields
\begin{equation}\label{eq:ghub}
g_\psi(x)=x^3+4x^5+2x^7.
\end{equation}
In particular, $g_\psi$ is strictly increasing and strictly convex on $(0,1)$ and satisfies $g_\psi(1)=7$.

If $G_n:=\psi^n(G)$ remains cubic for all $n\ge0$ and $x_n:=\mu(G_n)^{-1}$, then Theorem~\ref{thm:main2} gives
$x_{n-1}=g_\psi(x_n)$ for $n\ge1$.
Since $g_\psi$ is strictly convex with $g_\psi(0)=0$ and $g_\psi'(0)=0$, the equation $g_\psi(x)=x$ has a unique solution
$x_\ast\in(0,1)$, and the sequence $(x_n)$ is monotone and converges to $x_\ast$ (cf.\ Remark~\ref{rmk:iterate} below).
Numerically, $x_\ast\approx 0.6050003$, hence the limiting connective constant is $1/x_\ast\approx 1.65289$.
\end{example}

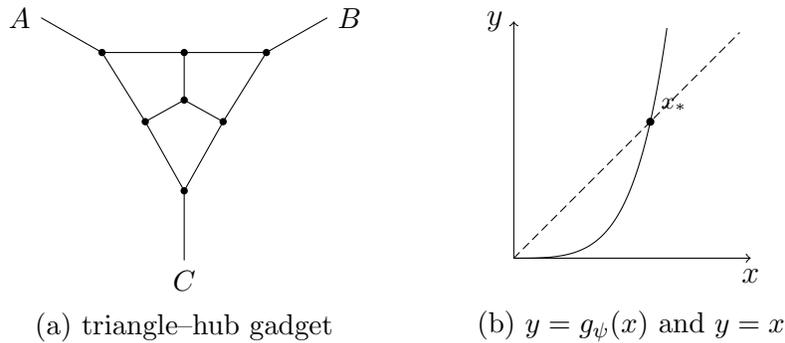
\begin{figure}[t]
\centering
\setlength{\tabcolsep}{18pt}
\renewcommand{\arraystretch}{1.0}
\begin{tabular}{cc}

\begin{tikzpicture}[scale=1.15, line cap=round, line join=round]
  \coordinate (Aout) at (-1.65, 0.95);
  \coordinate (Bout) at ( 1.65, 0.95);
  \coordinate (Cout) at ( 0.00,-1.85);

  \coordinate (a) at (-0.95, 0.55);
  \coordinate (b) at ( 0.95, 0.55);
  \coordinate (c) at ( 0.00,-1.05);

  \coordinate (ab) at ( 0.00, 0.55);
  \coordinate (bc) at ( 0.45,-0.25);
  \coordinate (ca) at (-0.45,-0.25);

  \coordinate (d) at (0.00, 0.00);

  \draw (a)--(ab)--(b)--(bc)--(c)--(ca)--(a);

  \draw (d)--(ab) (d)--(bc) (d)--(ca);

  \draw (a)--(Aout);
  \draw (b)--(Bout);
  \draw (c)--(Cout);

  \foreach \p in {a,b,c,ab,bc,ca,d} \fill (\p) circle (1.2pt);

  \node[left]  at (Aout) {\small $A$};
  \node[right] at (Bout) {\small $B$};
  \node[below] at (Cout) {\small $C$};

  \node[below=14pt] at (0,-1.9) {\small (a) triangle--hub gadget};
\end{tikzpicture}

&

\begin{tikzpicture}[scale=3.0, line cap=round, line join=round]
  \draw[->] (0,0) -- (1.05,0) node[below] {$x$};
  \draw[->] (0,0) -- (0,1.05) node[left] {$y$};

  \begin{scope}
    \clip (0,0) rectangle (1.02,1.02);

    \draw[densely dashed] (0,0) -- (1,1);

    \draw[domain=0:1, samples=220, smooth]
      plot (\x,{(\x)^3 + 4*(\x)^5 + 2*(\x)^7});
  \end{scope}

  \fill (0.6050,0.6050) circle (0.018);
  \node[above right] at (0.6050,0.6050) {\scriptsize $x_\ast$};

  \node[below=14pt] at (0.52,-0.02) {\small (b) $y=g_\psi(x)$ and $y=x$};
\end{tikzpicture}

\end{tabular}

\caption{(a) The triangle--hub gadget used to define the local transformation $\psi$.
(b) The fixed point $x_\ast$ solving $g_\psi(x)=x$ for $g_\psi(x)=x^3+4x^5+2x^7$. So $x_\ast=\sqrt{\frac{\sqrt3-1}{2}}\approx 0.6050003$,
and
$\mu(\psi^n(G))\longrightarrow \mu_\ast:=x_\ast^{-1}\approx 1.6528917$.}
\label{fig:triHubFixedPoint}
\end{figure}

\begin{remark}[Iterating a fixed transformation]\label{rmk:iterate}
Assume $\phi$ is a local transformation such that $G_n:=\phi^n(G)$ is cubic for all $n\ge0$.
Let $x_n:=\mu(G_n)^{-1}$ and $g(x):=g_\phi(x)$.
Then Theorem~\ref{thm:main2} gives
\[
x_{n-1}=g(x_n)\qquad(n\ge1),
\]
so $x_n=g^{-1}(x_{n-1})$.

Since $g$ is a polynomial with nonnegative coefficients and no constant/linear term
(in particular $g(0)=g'(0)=0$), it is strictly increasing on $(0,1]$ and hence invertible
on its image.

Assume in addition that $g$ is strictly convex on $(0,1)$ and that $g(1)\ge 1$.
Then $h(x):=g(x)-x$ is strictly convex with $h(0)=0$, $h'(0)=-1$, and $h(1)\ge0$,
so there exists a unique $x_\ast\in(0,1]$ such that $g(x_\ast)=x_\ast$.
Moreover, $(x_n)$ is monotone (increasing if $x_0<x_\ast$ and decreasing if $x_0>x_\ast$)
and converges to $x_\ast$.
\end{remark}

\section*{Acknowledgements}
The authors thank the organizers of the conference “Perfectly matched perspectives on statistical mechanics, combinatorics and geometry” (CIRM, 16–20 June 2025) for a stimulating meeting. The conference was supported by the National Science Foundation under Grant DMS-2512154. B.G.\ was partially supported by a U.S.\ Department of Education GAANN Fellowship.

\bibliography{fisher2}
\bibliographystyle{amsplain}

\end{document}